\documentclass[12pt]{article}
\usepackage{amssymb,amsmath,amsthm,amscd}

\newtheorem{df}{Definition}
\newtheorem{lem}{Lemma}
\newtheorem{teo}{Theorem}

\newtheorem{cor}{Corollary}
\newcommand{\N}{\mathbb{N}}
\newcommand{\Z}{\mathbb{Z}}
\newcommand{\C}{\mathbb{C}}
\newcommand{\R}{\mathbb{R}}
\newcommand{\p}{\mathbb{P}}
\newcommand{\nn}{\newline\noindent}
\newcommand{\z}{\xi}
\newcommand{\e}{\varepsilon}

\begin{document}
\title{{Convexity properties of coverings of 1-convex surfaces}
\author{M. Col\c toiu and C. Joi\c ta }}

\date{}
\maketitle
\begin{abstract}
We prove that there exists a 1-convex surface whose universal covering does not
satisfy the discrete disk property.
\end{abstract}

\section{Introduction}

The well-known Shafarevich Conjecture asserts that the universal covering space of a projective algebraic 
manifold is holomorphically convex. Although there are partial results, a complete 
answer to this problem is not known even for surfaces. (We remark that if instead of the universal 
covering one considers an arbitrary non-compact one, there are counterexamples, see \cite{Nap}). 

In this paper we are interested in studying convexity properties of the universal covering of 
1-convex surfaces. We recall that projective algebraic manifolds are a particular case of Moishezon manifolds, that
the exceptional set of a 1-convex manifold is a Moishezon space and that every Moishezon space is 
the exceptional set of a 1-convex manifold.

Suppose that $X$ is a 1-convex surface and $p:\tilde X\to X$ is a covering map. It is known (see \cite{Co}) 
that in general $\tilde X$ is not holomorphically convex. In fact $\tilde X$ might not be even
weakly 1-complete (that is, $\tilde X$ might not carry a continuous plurisubharmonic exhaustion function).
However $\tilde X$ can be exhausted by a sequence of strongly pseudoconvex domains and therefore $\tilde X$
satisfies the continuous disk property (see the next section for a precise definition). 
We investigate the \textit{discrete} disk property for $\tilde X$ which definitely is a stronger property.

Our main goal is to give an example of a 1-convex surface whose universal covering does not
satisfy the discrete disk property. In particular it will not be $p_5$-convex in the sense of \cite{DG}.
We proved in \cite{CJ} that if $\tilde X$ does not contain
an infinite Nori string of rational curves then actually $\tilde X$ does satisfy the discrete disk property.
Therefore our example must contain such a Nori string.

We remark that important convexity properties of coverings of 1-convex manifolds with respect to meromorphic
functions have been established in \cite{Nap}.

\section{Preliminaries}

We denote by $\Delta$ the unit disk in $\C$,  $\Delta=\{z\in\C: |z|< 1\}$ and for $c>0$
by $\Delta_{1+c}$ the disk $\Delta_{1+c}:=\{z\in\C:|z|<1+c\}$.
\nn
For $\epsilon>0$ we define $H_\epsilon\subset \C\times\R$ as $$H_\epsilon=
\Delta_{1+\epsilon}\times[0,1)\bigcup \{z\in\C:1-\epsilon<|z|<1+\epsilon\}\times\{1\}.$$

The following is just an intrinsic version of the classical Continuity Principle (see, for example
\cite{GF} page 47).

\begin{df}
A complex space $X$ is said to satisfy the continuous disk property if whenever $\epsilon$ is a positive number
and $F:H_{\epsilon}\to X$ is a continuous function such that, for every $t\in[0,1)$, 
$F_t:\Delta_{1+\epsilon}\to X$, $F_t(z)=F(z,t)$, is holomorphic we have that $F(H_{\epsilon_1})$
is relatively compact in $X$ for any $0<\epsilon_1<\epsilon$.
\end{df}

Motivated by the above definition we introduced in \cite{CJ}: 

\begin{df}
Suppose that $X$ is a complex space. We say that $X$ satisfies the discrete disk property if
whenever  $g_n:U\to X$ is a sequence of holomorphic functions defined on an open neighborhood 
$U$ of $\overline \Delta$ for which there exists an $\epsilon>0$ and a continuous function 
$\gamma:S^1=\{z\in\C: |z|=1\}\to X$ such that $\Delta_{1+c}\subset U$, 
$\bigcup_{n\geq 1} g_n(\Delta_{1+\epsilon}\setminus \Delta)$ is 
relatively compact in $X$ and ${g_n}_{|S^1}$ converges uniformly to $\gamma$ we have that 
$\bigcup_{n\geq 1}g_n(\overline \Delta)$ is relatively compact in $X$.
\end{df}

Note that if a complex space is $p_5$-convex in the sense of Docquier and Grauert \cite{DG} then it satisfies
the discrete disk property. Therefore our example will not be $p_5$-convex either. $X$ is called
$p_5$-convex if whenever $\{\Delta_\nu\}_{\nu\geq 0}$ is a sequence of holomorphic disks such that
$\bigcup_{\nu\geq 0}\partial\Delta_\nu\Subset X$ we have that 
$\bigcup_{\nu\geq 0}\overline\Delta_\nu\Subset X$ as well.

In \cite{Fo} it is constructed a complex manifold which is an increasing
union of Stein open subsets, and therefore it satisfies the continous disk property, but
it does not satisfy the discrete disk property. In particular this shows that the discrete disk property 
is stronger that the continuous one.

We recall that a compact complex curve is called rational if its normalization is ${\mathbb P}^1$.

A complex manifold is called 1-convex if it is the modification of a Stein space at a finite set of points.

\begin{df}
Let $L$ be a connected 1-dimensional complex space and $\cup L_i$ be its decomposition into 
irreducible components. $L$ is called an infinite Nori string if all $L_i$ are compact and $L$ is 
not compact
\end{df}

The following theorem was proved in \cite{CJ}.

\begin{teo}
Let $X$ be a 1-convex surface and $p:\tilde X\to X$ be a covering map. If $\tilde X$ does not contain
an infinite Nori string of rational curves then $\tilde X$ satisfies the discrete disk property.
\end{teo}

\section{The Results}

As we mentioned in the introduction, our goal is to prove the following theorem.

\begin{teo}\label{main}
There exists a 1-convex surface whose universal covering does not
satisfy the discrete disk property.
\end{teo}

We will describe first the basic idea of the proof of the theorem. We start with a basic example
of a 2-dimensional complex manifold $X$ that does not satisfy the discrete disk property and contains an 
infinite Nori string of rational curves. Let $f_n:\C\to\C^2$,
$f(\lambda)=((\frac{\lambda}2)^n,\lambda)$. We consider the complex manifold which is obtained from $\C^2$
after an infinite sequence of blow-ups. The center of a blow-up is chosen to be a point on the exceptional divisor of 
the previous blow-up. More precisely, for $k\leq n-1$ the center of the $k^{th}$ blow-up is the image of the origin
through the the proper transform of $f_n$. It is not difficult to see that the manifold thus obtained does not satisfy
the discrete disk property. It contains a Nori string $\{L_k\}_{k\geq 0}$ of curves isomorphic to $\p^1$.
Going backwards we can construct from $X$ a manifold $Y$ that contains a  Nori string 
$\{L_k\}_{k\in\Z}$ (so $L_k$ is defined also for $k<0$). Then
$\bigcup_{k\in\Z} L_k$ will cover $F_0\cup F_1$ where $F_0$ and $F_1$ are isomorphic to $\p^1$ and 
$F_0\cap F_1$ has exactly two points. An appropriately chosen neighborhood $U$ of 
$\bigcup_{k\in\Z} L_k$  in $Y$
will cover a manifold $V$ which is a neighborhood of  $F_0\cup F_1$. It is again not very hard to 
prove that $U$ 
does not  satisfy the discrete disk property. However  $F_0\cup F_1$ is not exceptional because 
the intersection matrix is
\[\left[
\begin{matrix}
 -2 & 2\\
2 & -2
\end{matrix}\right]
\]
\noindent
and then we 
have to blow-up again at two points, one on $F_0$ and one on $F_1$
(hence we have to blow-up $L_k$ accordingly) in order to make the intersection matrix 
negative defined. To show that the manifold $Z$ obtained in this way does 
not satisfy the discrete disk property is not easy anymore. A sequence of holomorphic disks defined in 
the simple-minded way as the one above will not work because their image will not stay in a small 
enough neighborhood in $Z$ of the proper transform of $\bigcup_{k\in\Z} L_k$. 

We move now to the proof of our theorem.
 
\

\noindent
{\bf Step 1.} We construct a 1-convex manifold $W$ and a covering $\tilde p:\tilde W\to W$. 
In the second step we will show that $\tilde W$ does not have the discrete disk property.

Suppose that $M$ is a two-dimensional complex manifold, $a\in M$ is a point and 
$(z_1,z_2):U\to \C^2$
is a local chart around $a$ such that $(z_1(a),z_2(a))=(0,0)$. We will define a complex manifold
${\cal I}(M,(z_1,z_2))$ as follows.

We assume first that $M=\Omega_0=\C^2$,  $(z_1,z_2)=(z_1^{(0)},z_2^{(0)})$ are the coordinate 
functions and $a=a_0=(0,0)$. 
We define $\Omega_k$ to be the manifold obtained from $\Omega_0$ by performing a sequence of 
$k$ successive 
blow-ups as follows. Let $\Omega_1$ be the blow-up of $\Omega_0$ in $a_0$. Namely 
$\Omega_1=\{(z_1^{(0)},z_2^{(0)},[\z_1^{(0)}:\z_2^{(0)}])\in 
\Omega_0\times\p^1:z_1^{(0)}\z_2^{(0)}=z_2^{(0)}\z_1^{(0)}\}$. 
Let $a_1=(0,0,[0:1])\in \Omega_1$ and $\Omega_2$ be the blow up of $\Omega_1$ in $a_1$
and let $L_0$ be the proper transform of the exceptional set of $\Omega_1$. The open subset of 
$\Omega_1$ given by $\z^{(0)}_2\neq 0$
is biholomorphic to $\C^2$ with the coordinate functions $z_1^{(1)}:=\frac{\z^{(0)}_1}{\z^{(0)}_2}$ 
and 
$z_2^{(1)}:=z_2^{(0)}$. In these coordinates $a_1$ is given by $z^{(1)}_1=0,\ z^{(1)}_2=0$. We 
continue this procedure 
$k$ times and we obtain $\Omega_k$.
In doing so we obtain also $L_0,\dots L_{k-1}$, which are complex curves each one of them isomorphic 
to $\p^1$, and 
$a_0,a_1,\dots, a_k$ the points where we are blowing up. Note that $\Omega_j\setminus \{a_j\}$ 
is an open subset of 
$\Omega_{j+1}\setminus\{a_{j+1}\}$. We set 
$${\cal I}(\Omega_0,(z_1^{(0)},z_2^{(0)})):=\cup _{j=0}^\infty \Omega_j\setminus \{a_j\}.$$
Notice that we have also a canonical map $\pi:{\cal I}(\Omega_0,(z_1^{(0)},z_2^{(0)}))\to\C^2$
such that $\pi^{-1}(0)=\bigcup_{k\geq 0}L_k$ and 
$\pi:{\cal I}(\Omega_0,(z_1^{(0)},z_2^{(0)}))\setminus \bigcup_{k\geq 0}L_k\to\C^2\setminus \{0\}$
is a biholomorphism. If $U$ is an open subset of $\C^2$ containing the origin we
set   ${\cal I}(U,(z_1,z_2)):=\pi^{-1}(U)$. Finally ${\cal I}(M,(z_1,z_2))$ is defined by gluing
$M\setminus\{a\}$ and  ${\cal I}(U,(z_1,z_2))$ via the identification of
$U\setminus\{a\}$ with ${\cal I}(U,(z_1,z_2))\setminus \bigcup_{k\geq 0}L_k$.
 
Notice that if $M$ is a complex manifold $\Omega$ is an open subset of $M$, $a$ is a point of 
$\Omega$
and $(z_1,z_2):U\to \C^2$ is a local chart around $a$, with $U\subset \Omega$ and  
$(z_1(a),z_2(a))=(0,0)$, then
$ {\cal I}(\Omega,(z_1,z_2))$ is an open subset of $ {\cal I}(M,(z_1,z_2))$

Next we will define inductively a sequence $\{X_k\}_{k\leq 0}$ of complex manifolds as follows. 
We consider $\C^2$ with coordinate functions  $(z_1^{(0)},z_2^{(0)})$ and we set 
$X_0={\cal I}(\C^2,(z_1^{(0)},z_2^{(0)}))$. To define $X_{-1}$
we let $M_{-1}$ be the blow-up of $\C^2$ at the origin, written in coordinates as follows:
$M_{-1}=\{(z_{1}^{(-1)},z_2^{(-1)},[\z_1^{(-1)}:\z_2^{(-1)}])\in
 \C^2\times\p^1:z_1^{(-1)}\z_2^{(-1)}=z_2^{(-1)}\z_1^{(-1)}\}$. Then
$\Omega_{-1}=\{(z_{1}^{(-1)},z_2^{(-1)},[\z_1^{(-1)}:\z_2^{(-1)}])\in M_{-1}:\z_2^{(-1)}\neq 0\}$ 
is an 
open set of $M_{-1}$, biholomorphic to $\C^2$ with coordinate functions $z_1^{(0)}:=
\frac{\z^{(-1)}_1}{\z^{(-1)}_2}$ and 
$z_2^{(0)}:=z_2^{(-1)}$. We set $X_{-1}={\cal I}(M_{-1},(z_1^{(0)},z_2^{(0)}))$. Notice then that 
$X_0$ is an open
subset of $X_{-1}$ and that $X_{-1}$ is biholomorphic to $X_0$. In the same way we defined $X_{-1}$ 
starting with $X_0$ 
we define $X_{k-1}$ starting with $X_k$. 

We put $X=\bigcup_{k=0}^{-\infty} X_k$ and $L=\bigcup_{k=-\infty}^{\infty} L_k$.
Notice that if $|j-k|\geq 2$ then $L_j\cap L_k=\emptyset$.

Next we want to define a fundamental system of open neighborhoods of $L_k$ for each $k\in\Z$.  
To do that we notice
that, by construction, $L_k$ is obtained as follows: we have 
$\C^2$ with coordinate functions $(z_1^{(k)},z_2^{(k)})$ we blow it up at the origin and
then we blow it up again at the point $(0,0,[0:1])$. The manifold thus obtained is denoted
by $\widehat\C^2$. Then $L_k$ is the proper transform of the exceptional
set of the first blow-up. That is we have that  $\widehat\C^2$ is given in 
$\C^2\times\p^1\times\p^1$ with coordinates
$(z_1^{(k)},z_2^{(k)},[\z_1^{(k)}:\z_2^{(k)}],[\z_1^{(k+1)}:\z_2^{(k+1)}])$ by

$$z_1^{(k)}\z_2^{(k)}=z_2^{(k)}\z_1^{(k)}, \z_1^{(k)}\z_2^{(k+1)}=
\z_1^{(k+1)}\z_2^{(k)}z_2^{(k)}
$$

In $\widehat\C^2$, $L_k$ is given by the equations $z_1^{(k)}=0$, $\z_2^{(k+1)}=0$.

For $r\in (0,1]$ we define $U_r^{(k)}:=\{|\z_2^{(k+1)}|<r|\z_1^{(k+1)}|,\ |z_1^{(k)}|<r\}$ 
and we notice that $\{ U_r^{(k)}\}_{r>0}$ 
is indeed a fundamental  system of open neighborhoods of $L_k$. Obviously $U_r^{(j)}$ and 
$U_r^{(k)}$
are biholomorphic for every $j$ and $k$.

We want to show that if $|j-k|\geq 2$ then $U_r^{(j)}\cap U_r^{(k)}=\emptyset$. It is clear 
from our construction
 that without loss of generality we can assume that $j=0$ and $k\geq 2$. As $U_r^{(j)}\cap U_r^{(k)}$
is an open set, it suffices to
 show that  $(U_r^{(0)}\setminus L)\cap (U_r^{(k)}\setminus L)=\emptyset$. 
We recall that we have defined $ z_1^{(k+1)}=\frac{\z_1^{(k)}}{\z_2^{(k)}}$ and 
$z_2^{(k+1)}=z_2^{(k)}$. Hence,
 outside $L$ and for $k\geq 0$, we have
that  $[z_1^{(k+1)}:z_2^{(k+1)}]=[\z_1^{(k)}:\z_2^{(k)}z_2^{(k)}]=
[z_1^{(k)}:z_2^{(k)}z_2^{(0)}]$. Inductively we get
$[z_1^{(k+1)}:z_2^{(k+1)}]=[z_1^{(0)}:(z_2^{(0)})^{k+2}]$. 
The inequality  $|z_1^{(k)}|<r$ is equivalent to $|\z_1^{(k-1)}|<r|\z_2^{(k-1)}|$.
As  $[\z_1^{(j)}:\z_2^{(j)}]= [z_1^{(j)}:z_2^{(j)}]$ for every $j\in \Z$ and every point in 
$X\setminus L$ it follows that
$$ U_r^{(k)}\setminus L= \{(z_1^{(0)}, z_2^{(0)})\in\C^2:
\ |z_2^{(0)}|^ {k+2}<r|z_1^{(0)}|, |z_1^{(0)}|< r|z_2^{(0)}|^ k\}.$$

We have that $ U_r^{(0)}\setminus L= 
\{(z_1^{(0)}, z_2^{(0)})\in\C^2:\ |z_2^{(0)}|^ {2}<r|z_1^{(0)}|, |z_1^{(0)}|< r\}.$ 
In particular every point of  $ U_r^{(0)}\setminus L$ satisfies $|z_2^{(0)}|^ {2}<r|z_1^{(0)}|<r^2$, 
hence $|z_2^{(0)}|<r$.
Then
a point in the intersection   $(U_r^{(0)}\setminus L)\cap (U_r^{(k)}\setminus L)$ would satisfy 
$|z_2^{(0)}|^ {2}<r|z_1^{(0)}|<
r^2|z_2^{(0)}|^ k$. As $k\geq 2$ we get $1<r^2|z_2^{(0)}|^ {k-2}<r^k$ and this contradicts 
our choice of $r\leq 1$. 

It is clear that the mapping $(z_1^{(k)},z_2^{(k)},[\z_1^{(k)}:\z_2^{(k)}],
[\z_1^{(k+1)}:\z_2^{(k+1)}])\rightarrow
(z_1^{(j)},z_2^{(j)},[\z_1^{(j)}:\z_2^{(j)}],[\z_1^{(j+1)}:\z_2^{(j+1)}])$ induces a 
biholomorphism of 
$p_{k,j}: U_r^{(k)}
\rightarrow U_r^{(j)}$. Let $Y:=U_1^{(0)}\cup U_1^{(1)}/\sim$, where $\sim$ identifies 
$U_1^{(2)}\cup U_1^{(1)}$ to $U_1^{(-1)}\cup U_1^{(0)}$ via $p_{2,0}$. 
Let ${\cal U}= \bigcup_{k\in\Z} U_1^{(k)}$
and $p:{\cal U}\rightarrow Y$ be the map given by $p_{|U_1^{(2k)}}=p_{2k,0}$ and
$p_{|U_1^{(2k+1)}}=p_{2k+1,1}$. Clearly $p$ is a covering map. $F_0:=p(L_0)$ and $F_1:=p(L_1)$ 
are both biholomorphic
to $\p^1$ and, moreover, we have $F_0\cdot F_0=-2$, $F_1\cdot F_1=-2$, 
$F_0\cdot F_1=2$. Let $\alpha_k\in L_k$ be the point given by
$(z_1^{(k)},z_2^{(k)},[\z_1^{(k)}:\z_2^{(k)}],[\z_1^{(k+1)}:\z_2^{(k+1)}])= (0,0,[1:1],[1:0])$ and 
$\beta_0,\beta_1\in Y$
the points $\beta_0=p(\alpha_{2k})$, $\beta_1=p(\alpha_{2k+1})$. We let $\pi: \tilde Y\to Y$ to be the blow up of 
$Y$ at $\beta_0$ and
$\beta_1$ and we denote by $\tilde F_0$ and $\tilde F_1$ respectively the proper transforms of 
$F_0$ and $F_1$.
Note that $\tilde F_0\cdot \tilde F_0=-3$, $\tilde F_1\cdot\tilde  F_1=-3$, $\tilde F_0\cdot\tilde  F_1=2$. 
As the intersection matrix 
\[\left[
\begin{matrix}
 -3 & 2\\
2 & -3
\end{matrix}\right]
\]
\noindent
 is negative
definite, it follows, see \cite{Gr}, that $\tilde F:=\tilde F_0\cup \tilde F_1$ is exceptional. We consider the 
following diagram:
$$\begin{CD}
\tilde{\cal U} @>\tilde \pi >> {\cal U}\\
 @V{\tilde p}VV   @V{p}VV\\
\tilde Y@>\pi >>Y
\end{CD}$$

We let $\tilde p:\tilde{\cal U}\rightarrow \tilde Y$ be the pull-back of $p$. Clearly 
$\tilde p$ is a covering map
and $\tilde \pi:\tilde{\cal U}\to{\cal U}$ is obtained by blowing-up $\cal U$ at every $\alpha_k$, $k\in\Z$. 
We choose now $W$ a 1-convex 
neighborhood of $\tilde F$ and we put $\tilde W:=\tilde p^{-1}(W)$, $\tilde L:=\tilde p^{-1}(\tilde F)$.
If $\tilde L_k$ is the proper transform of $L_k$ in $\tilde{\cal U}$ then $\tilde L=\bigcup \tilde L_k$. 
We will show that $\tilde W$ does not have the discrete disk property. In our construction of the sequence
of holomorphic discs we want to make sure that their image stays in $\tilde W$. To do that we need a ``concrete''
open neighborhood of $\tilde L$ in $\tilde W$. To obtain it  we consider $\{\tilde W^{(k)}_{r,\rho}\}$ 
a fundamental system of neighborhoods for $\tilde L_k$, each one of them being actually the preimage via 
$\tilde \pi$ of a cone centered at $\alpha_k$. Moreover $p_{k,j}$ induces a biholomorphism $\tilde W^{(k)}_{r,\rho}
\to\tilde W^{(j)}_{r,\rho}$. The construction is as follows.

We have the following description
of the blow-up of $U^{(k)}_1$ in $\alpha_k$: it is the set $\tilde U^{(k)}_1$ of all
$$(z_1^{(k)},z_2^{(k)},[\z_1^{(k)}:\z_2^{(k)}],[\z_1^{(k+1)}:\z_2^{(k+1)}],[w_1:w_2])
\in \C^2\times\p^1\times\p^1\times\p^1$$
such that 
$$z_1^{(k)}\z_2^{(k)}=z_2^{(k)}\z_1^{(k)},\ \z_1^{(k)}\z_2^{(k+1)}=
\z_1^{(k+1)}\z_2^{(k)}z_2^{(k)},\ w_2 z_1^{(k)}\z_1^{(k)}=w_1(\z_1^{(k)}-\z_2^{(k)})$$
and $$|z_1^{(k)}|<1,\
|\z_2^{(k+1)}|<|\z_1^{(k+1)}|$$
The proper transform of $L_k$ is given by $z_1^{(k)}=0$, $\z_2^{(k+1)}=0$, $w_1=0$.
A fundamental system of neighborhoods for $\tilde L_k$ is given by $$\tilde W^{(k)}_{r,\rho}=
\{|z_1^{(k)}|<r,\
|\z_2^{(k+1)}|<r|\z_1^{(k+1)}|,\ |w_1|<\rho|w_2|\}\subset\tilde U^{(k)}_1. $$
There exist then $\rho>0$ and $r>0$ such that $\tilde W_{r}^\rho=\bigcup_{k\in\Z} 
\tilde W^{(k)}_{r,\rho}\subset \tilde W $.
\nn
If we denote by $W_{r}^\rho\subset {\cal U}$ the set
$$ \bigcup_{k\in \Z}\{(z_1^{(k)},z_2^{(k)},[\z_1^{(k)}:\z_2^{(k)}],
[\z_1^{(k+1)}:\z_2^{(k+1)}])\in U_r^{(k)}: 
|z_1^{(k)}\z_1^{(k)}|<\rho |\z_2^{(k)}-\z_1^{(k)}|\}$$
\noindent
we have that $\tilde W\setminus\tilde L\supset \tilde W_{r}^\rho\setminus \tilde L 
\supset W_{r}^\rho\setminus L$. We notice at the same time
that keeping $\rho\in(0,1)$ fixed and choosing a small enough $r>0$ we have that 
$\tilde W^{(k)}_{r,\rho}\cap \tilde W^{(k+1)}_{r,\rho}=U_r^{(k)}\cap U_r^{(k+1)}$ for 
every $k\in\Z$. We fix such an $r\in(0,1)$ that satisfies also $r\leq \frac{\rho}2(1-r)$.

\
\nn
{\bf Step 2.} We construct a sequence of holomorphic discs that proves that $\tilde W$
does not have the discrete disk property. 

We fix $n\in\N$. To define our $n^{th}$ holomorphic disk, $g_n$, we will start with two 
polynomial functions $f_1=f_1^{(n)}$ and 
$f_2=f_2^{(n)}$ and $g_n$ will be the proper transform of $(f_1,f_2):\C\to\Omega_0$ 
restricted to a neighborhood of  $\overline\Delta_2$ (we recall that $\Omega_0$ was defined as $\C^2$ with coordinate 
functions $(z_1^{(0)},z_2^{(0)})$). This proper transform is considered after all the blow-ups we made, i.e. 
first at the points $\{a_j\}_{j\in\Z}$ and then $\{\alpha_j\}_{j\in\Z}$.

Let $c_1,\dots, c_{n-1}$ be integers defined recursively by $c_1=1$ and, for $k\geq 2$, 
$c_k=2k-1+(k-1)c_1+(k-2)c_2+\cdots c_{k-1}$. We also consider $d_1,\dots,d_{n-1}$ positive
integers defined by $d_{n-1}=1$ and, for $k\leq n-2$, $d_k=d_{k+1}+2d_{k+2}+\cdots
(n-k-1)d_{n-1}+ n-k$.  Let $N=2n(d_1+d_2+\dots + d_{n-1}+1)$.

We define $f_1$ and $f_2$ as
$$f_1(\lambda)=\e P_1(\lambda)P_2^2(\lambda)\cdots P_{n-1}^{n-1}(\lambda)\cdot\lambda^n,$$  
$$f_2(\lambda)=\e^2 P_1(\lambda)P_2(\lambda)\cdots P_{n-1}(\lambda)\cdot\lambda$$
where $\e$ is a positive real number that satisfies $\e<(\frac 16)^N\frac 1{n+2}r$ and
$P_1,\dots,P_{n-1}$ are polynomials defined recursively by 
$P_{n-1}(\lambda)=\e^{c_{n-1}}-\lambda$ and, for $k\leq n-2$,
$P_k(\lambda)=\e^{c_k}-
P_{k+1}(\lambda)\cdot P_{k+2}^2(\lambda)\cdots P_{n-1}^{n-k-1}(\lambda)\cdot \lambda^{n-k}$

\
\nn
{\bf Remarks:} 1) $P_k(0)\neq 0$ and $P_j$ and $P_k$ have no common zero for $j\neq k$. 
 \nn
2) Each $P_k$ is a monic polynomial of degree $d_k$. 

\

There are four conditions that we want the sequence 
$\{g_n\}$ to satisfy:
\nn
I) $g_n(\overline \Delta_2)\subset \tilde W$. We will prove in fact that $g_n(\overline \Delta_2)\subset 
\tilde W_r^\rho$.
\nn
II) $\bigcup_{n\geq 1} g_n(\Delta_{2}\setminus \Delta)$ is 
relatively compact in $\tilde W$ 
\nn
III) ${g_n}_{|S^1}$ is uniformly convergent 
\nn
IV) $\bigcup_{n\geq 1}g_n(\overline \Delta)$ is not relatively compact in $\tilde W$.

\
\nn
$\bullet$ Because $P_k(0)\neq 0$, the definition of $f_1$ and $f_2$ implies that the origin $0\in\C$ is a zero
of order $1$ for $f_2$ and a zero of order $n$ for $f_1$. This implies that
$g_n(0)\in L_{n-1}$ and this shows that $\{g_n(0)\}_{n\geq 1}$ is not relatively compact in $X$.
Hence  $\{g_n\}$ satisfies property IV).

\
\nn
$\bullet$ We will prove next that $\{g_n\}$ satisfies properties II) and III). 
\nn
Let $K_n:=
\{(z_1,z_2,[\z_1:\z_2]\in\Omega_1:|z_1|\leq\frac 1n,\ |z_2|\leq\frac 1n,\ 
|\z_2|\leq\frac 1n |\z_1|\}$. Note that $K_n$ is a compact subset of $X$,
$K_n\supset K_{n+1}$, and $\cap_{n\geq 1} K_n=\{(0,0,[1:0])\}$. Hence for $n$ large enough
$K_n\subset \tilde W$. Therefore if we show that 
$g_n(\{\lambda\in\C:1\leq |\lambda|\leq 2\})\subset K_n$
then we will prove both I) and II). 

\begin{lem}\label{roots}
For $k\in\{1,\dots,n-1\}$,  if $P_k(\lambda)=0$ then 
$|\lambda|<\frac{1}{2^k}$.
\end{lem}
\begin{proof}
We will prove our assertion by backward induction on $k$. For $k=n-1$ the statement is obvious.
We assume that we have proved our assertion for $j\geq k+1$ and we prove it for $k$.
For $j\geq k+1$ as $P_j$ are monic polynomials and all they zeros are inside the disk 
$\{|\lambda\in\C:|\lambda|<\frac 1{2^j}\}\subset \{\lambda\in\C:|\lambda|<\frac 1{2^{k+1}}\}$
we have that, for every $\lambda\in\C$ with $|\lambda|=\frac 1{2^k}$,
$|P_j(\lambda)|\geq(\frac 12)^{d_j(k+1)}$ (see for example the proof of the next Corollary). It follows that 
$|P_{k+1}(\lambda)\cdot P_{k+2}^2(\lambda)\cdots P_{n-1}^{n-k-1}(\lambda)\cdot \lambda^{n-k}|
\geq \frac 1{2^N}>\e>\e^{c_k}$
for  $|\lambda|=\frac 1{2^k}$. Rouch\'e's theorem (see e.g. \cite{Na} page 106) implies that $P_k(\lambda)$ and 
$P_{k+1}(\lambda)\cdot P_{k+2}^2(\lambda)\cdots P_{n-1}^{n-k-1}(\lambda)\cdot \lambda^{n-k}$
have the same number of zeros inside the disk $\{\lambda\in\C:|\lambda|<\frac 1{2^k}\}$. As the two 
polynomials have the same degree and all the zeros of the second one are in this disk, it
follows that all the zeros of $P_k$ are in there as well.
\end{proof}

\begin{cor}\label{bounds}
If $\lambda\in\C$ satisfies $1\leq |\lambda|\leq 2$
then $(\frac 12)^{d_k}< |P_k(\lambda)|< 3^{d_k}$
\end{cor}
\begin{proof}
Because $P_k$ is a monic polynomial of degree $d_k$ we have that it is of the form 
$P_k(\lambda)=(\lambda-\lambda^{(k)}_1)\cdots (\lambda-\lambda^{(k)}_{d_k}) $
where $\lambda^{(k)}_j$ are its roots (counted with multiplicity). Lemma \ref{roots} implies that 
$|\lambda^{(k)}_j|<\frac{1}{2^k}\leq\frac 12$ and therefore for $1\leq |\lambda|\leq 2$ we have that
$\frac 12<|\lambda^{(k)}_j-\lambda|<2+\frac 12<3$.
  
\end{proof}

Given our choice of $\e$ and Corollary \ref{bounds}, a simple computation shows:
\begin{cor}\label{ineq}
If $\lambda\in\C$ satisfies $1\leq |\lambda|\leq 2$
then we have:
\nn
a) $|f_1(\lambda)|<\frac 1n r\leq \frac 1n$, 
\nn
b) $|f_2(\lambda)|<\frac 1n r^2\leq \frac 1n$,
\nn
c) $|f_2(\lambda)|<\frac 1n |f_1(\lambda)|$,
\nn
d) $|f_1(\lambda)|>|f_2(\lambda)|^k$ for every $k\geq 1$.
\end{cor}

As $f_1$ and $f_2$ have no zero inside 
$\{\lambda\in\C:1\leq| \lambda|\leq 2\}$ this last Corollary implies that 
$g_n(\{\lambda\in\C:1\leq |\lambda|\leq 2\})\subset K_n$.

$\bullet$ We move now to the proof of property I). 
\nn
Let ${\cal Z}=\{0\}\cup\{\lambda\in\C: \exists k \text{ such that }P_k(\lambda)=0\}$ and we 
make the obvious remark that $f_1(\lambda)=0$ if and only if  $f_2(\lambda)=0$
if and only if $\lambda\in{\cal Z}$. 

We we will show first that $(f_1,f_2)(\overline\Delta_2\setminus {\cal Z})\subset 
\bigcup_{k\geq 0}U_r^{(k)}\setminus L\subset 
{\cal U}\setminus L$.
\nn
We have seen that $ U_r^{(k)}\setminus L= \{(z_1^{(0)}, z_2^{(0)})\in\C^2:
|z_1^{(0)}|< r|z_2^{(0)}|^ k,\ |z_2^{(0)}|^ {k+2}<r|z_1^{(0)}|\}.$ We prove that
$\bigcup_{k\geq 0} U_r^{(k)}\setminus L
\supset \{(z_1^{(0)}, z_2^{(0)})\in\C^2:
|z_1^{(0)}|< r,\ |z_2^{(0)}|< r^2\} \setminus\{(z_1^{(0)}, z_2^{(0)})\in\C^2:z_1^{(0)}=0\}$. 
This inclusion together with Corollary \ref{ineq} implies that indeed 
$(f_1,f_2)(\overline\Delta_2\setminus {\cal Z})\subset {\cal U}\setminus L$. 
Let $z_1^{(0)}, z_2^{(0)}\in\C$ be such that $0<|z_1^{(0)}|<r$ and $|z_2^{(0)}|<r^2$. 
If $z_2^{(0)}=0$ then obviously $(z_1^{(0)}, z_2^{(0)})\in U_r^{(0)}\setminus L$. Otherwise
we notice that $\frac{ |z_2^{(0)}|^{k+2}}{r}< r |z_2^{(0)}|^k$ (we have assumed that $r<1$) and we
let $I_k:=(\frac{ |z_2^{(0)}|^{k+2}}{r}, r |z_2^{(0)}|^k)\subset\R$. As 
$\frac{ |z_2^{(0)}|^{k+2}}{r}< r |z_2^{(0)}|^{k+1}$ it follows that $I_k\cap I_{k+1}\neq\emptyset$.
At the same time $I_0=(\frac{|z_2^{(0)}|^2}{r}, r)$ and $\lim_{k\to\infty}\frac{ |z_2^{(0)}|^{k+2}}{r}=0$.
This implies that $\bigcup_{k\geq 0} I_k=(0,r)$ and therefore $|z_1^{(0)}|\in \bigcup_{k\geq 0} I_k$.

Moreover we claim that $(f_1,f_2)(\overline\Delta_2\setminus {\cal Z})\subset 
\bigcup_{k= 0}^{n-1} U_r^{(k)}\setminus L$. To prove this it is enough to show that for $k\geq n$ one has
$|f_1(\lambda)|\geq r|f_2(\lambda)|^k$. However from Corollary \ref{ineq}, d) we have that
$|f_1(\lambda)|>|f_2(\lambda)|^k> r |f_2(\lambda)|^k$ for  $1\leq |\lambda|\leq 2$. As 
$\frac{f_2(\lambda)^k}{f_1(\lambda)}$ is a holomorphic function for $k\geq n$, the maximum modulus principle
implies that the inequality is valid on $\overline\Delta_2$.

Notice now that outside $L$ the inequality $|z_1^{(k)}\z_1^{(k)}|<\rho |\z_2^{(k)}-\z_1^{(k)}|$ is 
equivalent to
$|z_1^{(k)}|^2<\rho |z_2^{(k)}-z_1^{(k)}|$. 
As $[z_1^{(k)}:z_2^{(k)}]=[z_1^{(0)}:(z_2^{(0)})^{k+1}]$ the last inequality
is equivalent to $|z_1^{(0)}|^2<\rho |(z_2^{(0)})^{k+1}-z_1^{(0)}|\cdot|z^{(0)}_2|^k$. 
Given the description of $U_r^{(k)}\cap W_r^\rho$ obtained above it suffices then to show that for every 
$\lambda\in\overline\Delta_2
\setminus{\cal Z}$ that satisfies $|f_1(\lambda)|<r|f_2(\lambda)|^k$ and
$|f_2(\lambda)|^{k+2}<r|f_1(\lambda)|$ and every $k$ with $0\leq k\leq n-1$ we have
 $|f_1(\lambda)|^2<\rho|f_2(\lambda)^{k+1}-f_1(\lambda)|\cdot|f_2(\lambda)|^k$. 
We will distinguish two cases: $k\geq 1$ and $k=0$.
For $k\geq 1$ let
$A_k=\{\lambda\in \Delta_2: |f_1(\lambda)|<r|f_2(\lambda)|^k\}$ which is an open
subset of $\C$. Notice that due to Corollary \ref{ineq} we have that $A_k$ is relatively compact in $\Delta_2$
and therefore on $\partial A_k$ we have that $|f_1(\lambda)|=r|f_2(\lambda)|^k$.
Moreover, if $j\leq k-1$ and $P_j(\lambda)=0$, then $\lambda\not\in\overline A_k$. Hence $\frac{1}{P_j}$ is holomorphic on
a neighborhood of $\overline A_k$

\begin{lem}\label{div}
$P_k(\lambda)$ is a divisor of $\e^{2k-1}\cdot P_1^{k-1}(\lambda)\cdot P_2^{k-2}(\lambda)\cdots P_{k-1}(\lambda)- 
P_{k+1}(\lambda)\cdot P_{k+2}^2(\lambda)\cdots
P_{n-1}^{n-k-1}(\lambda)\cdot \lambda^{n-k}$.

\end{lem}

\begin{proof}
For $k=1$ we have to show that $P_1(\lambda)$ is a divisor of 
$\e -P_{2}(\lambda)\cdots P_{n-1}^{n-2}(\lambda)\cdot \lambda^{n-1}$. However, 
by definition $c_1=1$ and hence
$P_1(\lambda)=\e -P_{2}(\lambda)\cdots P_{n-1}^{n-2}(\lambda)\cdot \lambda^{n-1}$ and 
therefore there is nothing to prove. Suppose that $k\geq 2$. Notice that for $j\leq {k-1}$ we have 
$P_j\equiv \e^{c_j}\ ({\rm mod}\ P_k)$. It follows that
$\e^{2k-1}\cdot P_1^{k-1}\cdot P_2^{k-2}\cdots P_{k-1}- 
P_{k+1}\cdot P_{k+2}^2\cdots
P_{n-1}^{n-k-1}\cdot \lambda^{n-k}\equiv \e^{2k-1}\cdot \e^{(k-1)c_1}\cdots \e^{c_{k-1}}
- P_{k+1}\cdot P_{k+2}^2\cdots
P_{n-1}^{n-k-1}\cdot \lambda^{n-k}\ ({\rm mod}\ P_k)$. However 
$2k-1+(k-1)c_1+(k-2)c_2+\cdots c_{k-1}=c_k$ and therefore 
$\e^{2k-1}\cdot P_1^{k-1}\cdot P_2^{k-2}\cdots P_{k-1}- 
P_{k+1}\cdot P_{k+2}^2\cdots
P_{n-1}^{n-k-1}\cdot \lambda^{n-k}\equiv \e^{c_k}-P_{k+1}\cdot P_{k+2}^2\cdots
P_{n-1}^{n-k-1}\cdot \lambda^{n-k}\equiv 0\ ({\rm mod}\ P_k)$.
\end{proof}

\begin{lem} \label{ml}
$|f_1(\lambda)|^2\leq \frac{\rho}2|f_2(\lambda)^{k+1}-f_1(\lambda)|\cdot|f_2(\lambda)|^k$
 for every  $\lambda\in\overline A_k$ and every $k$ with $1\leq k\leq n-1$.
\end{lem}

\begin{proof}

We claim that on a neighborhood of $\overline A_k$ the meromorphic function
$$\frac{f_1^2(\lambda)}{(f_2^{k+1}(\lambda)-f_1(\lambda))\cdot f_2^k(\lambda)}$$ is actually holomorphic.
We consider first the case $k\leq n-2$ and we notice that 

\
\nn
$f_2^{k+1}(\lambda)-f_1(\lambda)=
\e P_1(\lambda)\cdot P_2^2(\lambda)\cdots P_{k+1}^{k+1}(\lambda)\cdot P_{k+2}^{k+1}(\lambda)\cdots  P_{n-1}^{k+1}(\lambda)\cdot\lambda^{k+1}\big(\e^{2k+1}\cdot P_1^{k}(\lambda)\cdot P_2^{k-1}(\lambda)\cdots 
P_{k}(\lambda)-P_{k+2}(\lambda)\cdot P_{k+3}^2(\lambda)\cdots
P_{n-1}^{n-k-2}(\lambda)\cdot \lambda^{n-k-1}\big)$.

We have seen that all zeros of $P_{k+2}\cdot P_{k+3}^2\cdots
P_{n-1}^{n-k-2}\cdot \lambda^{n-k-1}$ are inside the disk $\{\lambda\in\C:|\lambda|<\frac 12\}\subset \Delta_2$.
At the same time from the definition of $\e$ and Corollary \ref{bounds} it follows that on 
$\{\lambda\in\C:1\leq|\lambda|\leq 2\}$ we have $|\e^{2k+1}\cdot P_1^{k}\cdot 
P_2^{k-1}\cdots P_{k}|<
|P_{k+2}\cdot P_{k+3}^2\cdots
P_{n-1}^{n-k-2}\cdot \lambda^{n-k-1}|$. Rouch\'e's theorem implies that 
$\e^{2k+1}\cdot P_1^{k}\cdot P_2^{k-1}\cdots P_{k}- 
P_{k+2}\cdot P_{k+3}^2\cdots
P_{n-1}^{n-k-2}\cdot \lambda^{n-k-1}$ has exactly $d_{k+2}+2d_{k+3}+\cdots+(n-k-1)d_{n-1}+n-k-1=d_{k+1}$
zeros inside $\Delta_2$.
Then Lemma \ref{div} implies that $\e^{2k+1}\cdot P_1^{k}\cdot P_2^{k-1}\cdots P_{k}- 
P_{k+2}\cdot P_{k+3}^2\cdots
P_{n-1}^{n-k-2}\cdot \lambda^{n-k-1}=P_{k+1}Q$ where $Q$ is a polynomial which is nonvanishing on a neighborhood
of $\overline \Delta_2$. We have seen that on a neighborhood of $\overline A_k$ we have that 
$P_1\cdot P_2^2\cdots P_{k-1}^{k-1}$ is nonvanishing. Hence we it remains to show that
$$\frac{f_1^2(\lambda)}{P_k^k\cdot P_{k+1}^{k+1}
\cdot P_{k+2}^{k+1}\cdots P_{n-1}^{k+1}\cdot\lambda^{k+1}
\cdot P_{k+1}\cdot P_k^k\cdot P_{k+1}^k\cdots P_{n-1}^k\cdot\lambda^k}$$
is holomorphic and this follows from the definition of $f_1$.

For $k=n-1$ Rouche's theorem implies as above that $f_2^n-f_1=f_1\cdot Q_1$ where $Q_1=\e^{2n-1}P_1^{n-1}
\cdot P_2^{n-2}\cdots P_{n-1}-1$ is  
nonvanishing on a neighborhood of $\overline \Delta_2$. It remains to notice that $\frac{f_1}{f_2^{n-1}}$ is holomorphic 
on a neighborhood of $\overline A_{n-1}$ and our claim is proved.

The maximum modulus principle implies that it is enough to check our inequality on $\partial A_k$, hence we may assume that 
$|f_1(\lambda)|=r|f_2(\lambda)|^k$. Then it suffices to show that
$r^2|f_2(\lambda)|^{2k}\leq \frac{\rho}2(r|f_2(\lambda)|^k -|f_2(\lambda)|^{k+1})\cdot|f_2(\lambda)|^k$.
Therefore it is enough to show that $r^2\leq  \frac{\rho}2(r-|f_2(\lambda)|)$. We have seen in Corollary \ref{bounds}
that $|f_2(\lambda)|\leq r^2$. This means that it is enough to show that $r^2\leq \frac{\rho}2(r-r^2)$
and this follows from our choice of $r$.
\end{proof}

This Lemma takes care of the case $1\leq k\leq n-1$. It remains to deal with $k=0$. That means that we have to show 
that for every 
$\lambda\in\overline\Delta_2
\setminus{\cal Z}$ that satisfies $|f_1(\lambda)|<r$ and
$|f_2(\lambda)|^{2}<r|f_1(\lambda)|$ we have
$|f_1(\lambda)|^2<\rho|f_2(\lambda)-f_1(\lambda)|$. This follows from the next Lemma.

\begin{lem}
For every $\lambda\in\overline\Delta_2$ we have  
$|f_1(\lambda)|^2\leq \frac{\rho}2|f_2(\lambda)-f_1(\lambda)|$
\end{lem}
\begin{proof}
Exactly as in the proof of Lemma \ref{ml} we get that $\frac{f_1^2(\lambda)}{f_2(\lambda)-f_1(\lambda)}$
is holomorphic on a neighborhood of $\overline\Delta_2$. Hence we have to check the inequality only on $\partial\Delta_2$.
That is, it suffices to show that $|f_1|^2+\frac{\rho}{2}|f_2|\leq \frac{\rho}{2}|f_1|$ on $\partial\Delta_2$.
This follows from Corollary \ref{bounds} (note that the two terms appearing on the left-hand side of the inequality
contain $\e^2$ and the one on right contains $\e$).
\end{proof}

\
\nn
{\bf Step 3.} We show that the universal covering 
of $\tilde W$ (hence of $W$) does not satisfy the discrete disk property.

 We will show first that $\tilde W_{r}^\rho$ is simply connected. As each $\tilde W^{(k)}_{r,\rho}$ 
is simply connected, it suffices to show that 
$\tilde W^{(k)}_{r,\rho}\cap \tilde W^{(k+1)}_{r,\rho}=U_r^{(k)}\cap U_r^{(k+1)}$ is connected
for every $k\in\Z$. Note that for points in $U_r^{(k)}\cap U_r^{(k+1)}$ we have that $\z_2^{(k)}\neq 0$,
$\z_1^{(k+1)}\neq 0$, $\z_1^{(k+2)}\neq 0$. Hence $U_r^{(k)}\cap U_r^{(k+1)}\subset \C^2$ where the coordinate
functions on $\C^2$ are $x:=\frac{\z_1^{(k)}}{\z_2^{(k)}}$ and $y=\frac{\z_2^{(k+1)}}{\z_1^{(k+1)}}$.
In this coordinates we have the following: $z_2^{(k)}=z_2^{(k+1)}=xy$, $z_1^{(k)}=x^2y$, $z_1^{(k+1)}=x$,
$\frac{\z_2^{(k+2)}}{\z_1^{(k+2)}}=xy^2$. Therefore 
$$U_r^{(k)}\cap U_r^{(k+1)}=\{(x,y)\in\C^2:|y|<r,\ |x^2y|<r\}\cap \{(x,y)\in\C^2:|xy^2|<r,\ |x|<r\}.$$
If $|x|<r$ and $|y|<r$ then $|x^2y|<r^3<r$ and $|xy^2|<r^3<r$
because we have assumed that $r<1$. it follows that
$$U_r^{(k)}\cap U_r^{(k+1)}=\{(x,y)\in\C^2: |y|<r,\ |x|<r\}.$$
In particular $U_r^{(k)}\cap U_r^{(k+1)}$ is connected (even contractible).

We proved that $g_n(\overline \Delta_2)\subset \tilde W_r^\rho$. It follows that the universal cover
of $\tilde W$ (which contains $\tilde W_r^\rho$) does not satisfy the discrete disk property.

\vspace{1.0cm}
\begin{flushleft}
Mihnea Col\c toiu \newline
Institute of Mathematics of the Romanian Academy\newline
P.O. Box 1-764, Bucharest 014700\newline 
ROMANIA\newline
\emph{E-mail address}: Mihnea.Coltoiu@imar.ro

\

\

Cezar Joi\c ta \newline
Institute of Mathematics of the Romanian Academy\newline
P.O. Box 1-764, Bucharest 014700\newline 
ROMANIA\newline
\emph{E-mail address}: Cezar.Joita@imar.ro
\end{flushleft}

\end{document}